\documentclass[11pt]{amsart}
\usepackage[letterpaper,margin=1.2in]{geometry}
\usepackage{amsbsy,amsfonts,amsmath,amssymb,amscd,amsthm,mathrsfs}
\usepackage{subfigure,enumitem,graphicx,epsfig,color}
\usepackage[T1]{fontenc}
\usepackage[colorlinks=true,linkcolor=blue,citecolor=green]{hyperref}
\normalsize

\newtheorem{theorem}{Theorem}[section]
\newtheorem{lemma}[theorem]{Lemma}

\newtheorem{definition}[theorem]{Definition}

\newtheorem{remark}[theorem]{Remark}

\newtheorem*{theorem A}{Theorem A}
\newtheorem*{corollary B}{Corollary B}
\newtheorem*{theorem C}{Theorem C}
\newtheorem*{corollary D}{Corollary D}
\theoremstyle{definition}

\begin{document}
\title {On the measure-theoretic pressure associated with Feldman-Katok metric and max-mean metric}

\author{Zhongxuan Yang$^{*,1}$}
\address{College of Mathematics and Statistics, Chongqing University, Chongqing 401331, China}
\thanks{$^*$ Corresponding author and email: yzx@stu.cqu.edu.cn.}
\email{yzx@stu.cqu.edu.cn}

\author{Xiaojun Huang$^1$}
\address{College of Mathematics and Statistics, Chongqing University, Chongqing 401331, China}
\email{hxj@cqu.edu.cn}

\author{Jiajun Zhang$^1$}
\address{College of Mathematics and Statistics, Chongqing University, Chongqing 401331, China}
\email{zjj123@stu.cqu.edu.cn}



\begin{abstract} In this manuscript,  we present  modified Feldman-katok metric $d_{F K_{n}^{q}}$ and modified max-mean metric $\hat{d}_{n}^{q}$, then apply them to examine  the  measure-theoretic pressure and topological pressure. Subsequently, we extend Katok’s entropy formula to measure-theoretic pressure associated with  modified Feldman-katok metric $d_{F K_{n}^{q}}$, and obtain an equivalence between the topological pressure with respect to modified Feldman-Katok metric $d_{F K_{n}^{q}}$ and the classical topological pressure. In addition, we  similarly  generalize Katok’s entropy formula to measure-theoretic pressure associated with  modified max-mean metric $\hat{d}_{n}^{q}$.
\end{abstract}

\maketitle
\thispagestyle{empty}

\section{Introduction}
Entropy serves as a fundamental measure of the complexity inherent in dynamical systems, which is measured by the exponential growth rate of the account  of orbits differentiated  with limiting accuracy. Both measure-theoretic entropy and topological entropy play pivotal roles in quantifying this complexity. The interplay between these two measures is encapsulated in the well-established variational principle.  Kolmogorov pioneered the concept of measure-theoretic entropy for measure-preserving dynamical systems, considering it as an isomorphic invariant quantity \cite{kolomogorov1958new,kolmogorov1959entropy}. Subsequently, Adler, Konheim and McAndrew  \cite{adler1965topological} introduced topological entropy for topological dynamical systems using open covers and regarded it as a conjugate invariant quantity. Then, Bowen \cite{bowen1971entropy} and Dinaburg \cite{dinaburg1971connection} independently presented equivalent formulations based on separating sets and spanning sets. These developments have laid the foundation for understanding the intricate relationship between measure-theoretic entropy and topological entropy in the realm of dynamical systems.

Later on, inspired by the Gibbs state theory in statistical mechanics, Ruelle \cite{ruelle1982repellers} introduced the notion of topological pressure in dynamical systems, which is a non-trivial and natural generalisation of topological entropy.  Later, Walters \cite{walters1982introduction}  extended it to general topological dynamical systems.  Aiming to better study the theory related to topological entropy, Pesin and Pitskel \cite{pesin1984topological}  generalized Bowen’s definition of topological entropy from a viewpoint of dimension theory and defined the topological pressure on non-compact sets. Topological pressure, the variational principle, and equilibrium state theory play an important role in statistical mechanics, ergodic theory, and dynamical systems, and form a major part of the thermodynamic formalism \cite{cai2023feldman,zhang2023note,zhang2009variational,he2004definition,huang2019measure1}.

In recent years, a wide variety of metrics have provided us with more perspectives and horizons when studying the characteristics  of dynamical systems, especially in terms of topological entropy, topological pressure, weigthed topological pressure and  bounded complexity etc.  Huang, Li, Thouvenot, Xu and Ye \cite{huang2021bounded} studied the topological dynamical systems and ergodic theory of dynamical systems with bounded complexity relevant to three metrics: the Bowen metric $d_n$, the max-mean metric $\hat{d}_{n}$ and the mean metric $\bar{d}_{n}$.  It is shown that an invariant measure $\mu$ has bounded complexity relevant to $\hat{d}_{n}$ if and only if $\mu$ has bounded complexity relevant to $\bar{d}_{n}$ if and if and only if it has discrete spectrum. Within the exploration of Sarnak's conjecture, Huang, Wang, and Ye \cite{huang2019measure} introduced the measure complexity relevant to an invariant measure, akin to the concept introduced by Katok \cite{katok1980lyapunov}, albeit employing the mean metric rather than the Bowen metric, and  they showed that when replacing Bowen metric with  mean metric, the Katok's entropy formula for  ergodic measure is still valid. Afterwards, Huang, Wei, Yu and Zhou \cite{huang2022measure} stated that  Katok’s entropy formula  still holds when replacing the Bowen metric with the modified mean metric in the study of  measure complexity for  topological dynamical systems.  Gr{\"o}ger and J{\"a}ger \cite{groger2016some}  provided a notion of topological entropy for the dynamical system relevant to mean metric using separated sets and demonstrated an equivalence between  topological entropy relevant to mean metric and  classical topological entropy. More recently, Huang, Chen and Wang \cite{huang2019measure1,huang2018entropy}  further investigated  the Katok’s entropy formula and Brin-Katok formula of conditional entropy in mean metric, established Katok’s entropy formula for ergodic measures relevant to mean metric.  Zhang, Zhao and Liu \cite{zhang2023note} explored  measure-theoretic pressure and topological pressure associated with modified Bowen metric and modified mean metric, they extended Katok’s entropy formula to measure-theoretic pressure with modified Bowen metric and modified mean metric.

The Feldman-Katok metric was introduced in  \cite{dominik2017feldman}, as a topological counterpart of the edit distance  $\overline{f}$ which was presented by Feldman  \cite{feldman1976new} to study loosely Bernoulli system. Recently, García-Ramos and Kwietniak \cite{garcia2022topological} introduced FK-sensitivity and FK-continuity to delineate models of zero entropy loosely Bernoulli systems. They rigorously demonstrated that, for a minimal dynamical system, it is  either FK-continuity or FK-sensitivity.  Downarowicz, Kwietniak and {\L}{\k{a}}cka \cite{downarowicz2021uniform} introduced the concept of $\overline{f}$-pseudometric to finite-valued stationary stochastic processes, employing it to explore entropy rate. Cai and Li \cite{cai2023feldman,cai2023feldman1} investigated entropy formulas and topological pressure  with respect to  the Feldman-Katok metric. Recently, Xie, Chen and Yang \cite{xie2024weighted} introduced weighted topological entropy defined by the Feldman-Katok metric and demonstrated its equivalence to the weighted topological entropy associated with  Bowen metric.

In this manuscript, drawing inspiration from the works of  \cite{huang2019measure1,huang2021bounded,zhang2023note,cai2023feldman1,cai2023feldman}, we adopt  two metrics, namely the modified Feldman-Katok metric $d_{FK_{n}^{q}}$ and the modified max-mean metric $\hat{d}_{n}^{q}$, to redefine measure-theoretic pressure and topological pressure. Specifically,  Katok’s entropy formula is extended to measure-theoretic pressure associated with the aforementioned metrics,  and we establish the equivalences between the topological pressure computed with these metrics and the classical topological pressure. This manuscript is structured as follows: In Section 2  we introduce the notion of modified Feldman-Katok metrics and review related results. In Sections 3,4 we delve into the study of measure-theoretic pressure and topological pressure  associated with the modified Feldman-Katok metric $d_{FK_{n}^{q}}$.  Lastly,  in Section 5 we discuss the  measure-theoretic pressure and topological pressure associated with the modified max-mean metric.

\section{Preliminaries}

In this manuscript,  a topological dynamical system (TDS for short) refers to a pair  $(X, T) $, where  $X$  is a compact metric space with a metric $ d $ and  $T$  is a homeomorphism from  $X $ to itself. Let  $C(X, \mathbb{R})$  denote the set of all continuous functions of  $X $.  The sets of natural and real numbers are denoted by $\mathbb{N}$ and $\mathbb{R}$, respectively. We denote $\{0,1,\cdots, n-1\}$ by $[n]$   For every $n\in \mathbb{N}$.  For the sake of convenience, we declare that
all our basic notions and notations (such as invariant measure, entropy, topological pressure  etc.)
are as in the textbook \cite{walters1982introduction}, and we will only explain some of them.

Given a TDS  $(X, T) $, assume  that  $\mathscr{B}(X) $ is the $\sigma$-algebra of Borel subsets of  $X$. Let  $M(X) $ be the collection of all  Borel probability measures defined on the measurable space  $(X, \mathscr{B}(X))$. 
We say  $\mu \in M(X)$  is  $T$-invariant if  $\mu\left(T^{-1}(A)\right)=\mu(A)$  holds for any  $A \in \mathscr{B}(X) $. Denote by $ M(X, T)$  the collection of all $ T $-invarant  Borel probability measures defined on the measurable space  $(X, \mathscr{B}(X))$. In the weak$^{*}$ topology,  $M(X, T)$  is a nonempty compact convex set. We say  $\mu \in M(X, T)$  is ergodic if for any  $A \in \mathscr{B}(X)$  with  $T^{-1} A=A$, $ \mu(A)=0  $ or $ \mu(A)=1 $ holds. Denote by  $E(X, T)$  the collection of all ergodic measures on $ (X, T)$. It is well known that $ E(X, T)$  is the collection of all extreme points of  $M(X, T)$ and  $E(X, T)$  is nonempty.  We note that each  $\mu \in M(X, T)$  induces a measure preserving system (MPS for short) $\left(X, \mathcal{B}_{X}, \mu, T\right) $ (see \cite{kerr2016ergodic,walters1982introduction}).

\subsection{Measure-theoretic entropy}

A partition of $ X $ is a family of subsets of  $X $ with union $ X $ and all elements of the family are pairwise disjoint. Denote by  $\mathcal{P}_{X}$  the collection of all finite Borel partitions of $ X $. For any  $\alpha, \beta \in \mathcal{P}_{X}, \alpha $ is said to be finer than  $\beta$  (write $ \beta \preceq \alpha $ ) if each atom of  $\alpha $ is contained in some atom of $ \beta $. Let  $\alpha \vee \beta=\{A \cap B: A \in \alpha, B \in \beta\} $.
Let  $\mu \in {M}(X, T) $. Given  $\alpha \in \mathcal{P}_{X} $, define
$$
H_{\mu}(\alpha)=\sum_{A \in \alpha}-\mu(A) \log \mu(A)
$$

It is not hard to see that  $\left\{H_{\mu}\left(\bigvee_{i=0}^{n-1} T^{-i} \alpha\right)\right\}_{n=1}^{\infty}$  is a non-negative and subadditive sequence. Thus we define the measure-theoretic entropy of  $\mu$  relative to  $\alpha$  by
$$
h_{\mu}(T, \alpha)=\lim _{n \rightarrow \infty} \frac{1}{n} H_{\mu}\left(\bigvee_{i=0}^{n-1} T^{-i} \alpha\right)=\inf _{n \geq 1} \frac{1}{n} H_{\mu}\left(\bigvee_{i=0}^{n-1} T^{-i} \alpha\right) .
$$

The measure-theoretic entropy of  $\mu$ is defined  by
$$
h_{\mu}(T)=\sup _{\alpha \in \mathcal{P}_{X}} h_{\mu}(T, \alpha) .
$$

In the following,  we shall apply Feldman-Katok metric  to explore  measure-theoretic pressure and topological pressure.  Let us recall the notion of Feldman-Katok metric.  Let  $(X, T) $ be a TDS. Given  $x, y \in X, \delta>0$  and  $n \in \mathbb{N} $, define an  $(n, \delta) $-match of  $x$  and  $y$  to be an order preserving (i.e.  $\pi(i)<\pi(j)$  whenever  $i<j $ ) bijection $ \pi: D(\pi) \rightarrow R(\pi) $ such that $ D(\pi), R(\pi) \subset [n] $ and for every  $i \in D(\pi)$ then  $d\left(T^{i} x, T^{\pi(i)} y\right)<   \delta $. Let  $|\pi|$  be the cardinality of $ D(\pi) $. We set
$$
\mathscr{F}_{n, \delta}(x, y)=1-\frac{\max \{|\pi|: \pi \text { is an }(n, \delta) \text {-match of } x \text { and } y\}}{n}
$$
and
$$
\mathscr{F}_{\delta}(x, y)=\limsup _{n \rightarrow+\infty} \mathscr{F}_{n, \delta}(x, y) \text {. }
$$

\begin{definition}\cite{cai2023feldman1}
 The Feldman-Katok metric of  $(X, T)$ is defined by
$$
d_{F K_{n}}(x, y)=\inf \left\{\delta>0: \mathscr{F}_{n, \delta}(x, y)<\delta\right\},
$$
and
$$
d_{F K}(x, y)=\inf \left\{\delta>0: \mathscr{F}_{\delta}(x, y)<\delta\right\}$$
for any  $x, y \in X $.
\end{definition}

Next, we  introduce the notion of modified  Feldman-Katok metric to investigate the measure-theoretic pressure. Let  $(X, T) $ be a TDS. Given  $x, y \in X, \delta>0$  and  $n,q \in \mathbb{N}$,  define an  $(q,n, \delta) $-match of  $x$  and  $y$  to be an order preserving  bijection $ \pi: D(\pi) \rightarrow R(\pi) $ such that $ D(\pi), R(\pi) \subset [n] $ and for every  $i \in D(\pi)$  then  $d\left(T^{qi} x, T^{q\pi(i)} y\right)<   \delta $. We set
$$
\mathscr{F}_{n, \delta}^{q}(x, y)=1-\frac{\max \{|\pi|: \pi \text { is an }(q,n, \delta) \text {-match of } x \text { and } y\}}{n}
$$
and
$$
\mathscr{F}_{\delta}^{q}(x, y)=\limsup _{n \rightarrow+\infty} \mathscr{F}_{n, \delta}^{q}(x, y) \text {. }
$$

\begin{definition}
 The $q$-Feldman-Katok metric of  $(X, T)$ is defined by
	$$
	d_{F K_{n}^{q}}(x, y)=\inf \left\{\delta>0: \mathscr{F}_{n, \delta}^{q}(x, y)<\delta\right\},
	$$
	and
	$$
	d_{F K^{q}}(x, y)=\inf \left\{\delta>0: \mathscr{F}_{\delta}^{q}(x, y)<\delta\right\}$$
	for any  $x, y \in X $. When $q=1$, $d_{F K_{n}^{1}}(\cdot, \cdot)$ and $d_{F K^{1}}(\cdot, \cdot)$ are just $d_{F K_{n}}(\cdot, \cdot)$ and $d_{F K}(\cdot, \cdot)$ respectively.
\end{definition}

 Given  $\varphi \in C(X, \mathbb{R}), \mu \in M(X, T) $ and  $\varepsilon>0 $, write
 $
 \operatorname{Var}(\varphi, \varepsilon):=\sup \{|\varphi(x)-\varphi(y)|: d(x, y)<\varepsilon\}
 $,  $\|\varphi\|=\sup_{x\in X}|\varphi (x)|,$ and
$
S_{n}^{q} \varphi(x):=\varphi(x)+\varphi\left(T^{q} x\right)+\ldots+\varphi\left(T^{q(n-1)} x\right),
$
when $q=1$, we simply write $ S_{n}^{1}$ as $ S_{n}$.

Let
$$
S_{d_{F K_{n}^{q}}}(\varphi, \mu, \varepsilon)=\inf \left\{\sum_{i=1}^{k} e^{S_{n}^{q} \varphi\left(x_{i}\right)}, \text { s.t. } \mu\left(\bigcup_{i=1}^{k} B_{d_{F K_{n}^{q}}}\left(x_{i}, \varepsilon\right)\right)>1-\varepsilon\right\}
$$
where
$
B_{d_{F K_{n}^{q}}}(x, \varepsilon)=\left\{y \in X: d_{F K_{n}^{q}}(x, y)<\varepsilon\right\}
$
for any  $x \in X $. We define
$$
S_{n}^{FK}(\varphi, \mu, \varepsilon)=\inf _{q \geq 1} S_{d_{F K_{n}^{q}}}(\varphi, \mu, \varepsilon) .
$$

In order to obtain the main results below, we also need the modified Bowen measure $d_{n}^{q}$ and the modified mean measure $\bar{d}_{n}^{q}$,
$$
d_{n}^{q}(x, y)=\max _{0 \leq i \leq n-1} d\left(T^{q i} x, T^{q i} y\right)
$$
$$
S_{d_{n}^{q}}(\varphi, \mu, \varepsilon)=\inf \left\{\sum_{i=1}^{k} e^{S_{n}^{q} \varphi\left(x_{i}\right)}, \text { s.t. } \mu\left(\bigcup_{i=1}^{k} B_{d_{n}^{q}}\left(x_{i}, \varepsilon\right)\right)>1-\varepsilon\right\}
$$
where
$
B_{d_{n}^{q}}(x, \varepsilon)=\left\{y \in X: {d_{n}^{q}}(x, y)<\varepsilon\right\}
$
for any  $x \in X $. We define
$$
\tilde{S}_{n}(\varphi, \mu, \varepsilon)=\inf _{q \geq 1} S_{d_{n}^{q}}(\varphi, \mu, \varepsilon) .
$$
When $q=1$, we simply write $d_{n}^{1}(\cdot, \cdot)$ as $d_{n}(\cdot, \cdot)$.

For any $x, y \in X $,  define
$$
\bar{d}_{n}^{q}(x, y)=\frac{1}{n} \sum_{i=0}^{n-1} d\left(T^{q i} x, T^{q i} y\right).
$$
 For  $\varepsilon>0 $ and  $\varphi \in C(X, \mathbb{R}) $, let
$$
S_{\bar{d}_{n}^{q}}(\varphi, \mu, \varepsilon)=\inf \left\{\sum_{i=1}^{k} e^{S_{n}^{q} \varphi\left(x_{i}\right)} \text { s.t. } \mu\left(\bigcup_{i=1}^{k} B_{\bar{d}_{n}^{q}}\left(x_{i}, \varepsilon\right)\right)>1-\varepsilon\right\},
$$
where
$
B_{\bar{d}_{n}^{q}}(x, \varepsilon)=\left\{y \in X: \bar{d}_{n}^{q}(x, y)<\varepsilon\right\}
$
for any  $x \in X $. We define
$$
\bar{S}_{n}(\varphi, \mu, \varepsilon)=\inf _{q \geq 1} S_{\bar{d}_{n}^{q}}(\varphi, \mu, \varepsilon) .
$$
When $q=1$, we simply write $\bar{d}_{n}^{1}(\cdot, \cdot)$ as $\bar{d}_{n}(\cdot, \cdot)$.

It is worth noting that when $q=1$,  the measure-theoretic pressure version of Katok's entropy formula associated with  Bowen metric has been established as follows:

\begin{theorem}\cite{he2004definition}
 Let  $(X, T)$  be a TDS. For any  $\mu \in E(X, T)$  and  $\varphi \in C(X, \mathbb{R}) $, then
$$
h_{\mu}(T)+\int \varphi \mathrm{d} \mu=\lim _{\varepsilon \rightarrow 0} \liminf _{n \rightarrow \infty} \frac{1}{n} \log S_{d_{n}}(\varphi, \mu, \varepsilon)=\lim _{\varepsilon \rightarrow 0} \limsup _{n \rightarrow \infty} \frac{1}{n} \log S_{d_{n}}(\varphi, \mu, \varepsilon) .
$$
\end{theorem}

In fact, when $q=1$,  the measure-theoretic pressure version of Katok's entropy formula associated with  mean metric is still valid. Actually, by \cite[Lemma 2.1]{huang2019measure1} and  a similar approach in the proof of Theorem \ref{3.2}, we can get the following result.
\begin{theorem} \label{2.4}
 Let  $(X, T)$  be a TDS. For any  $\mu \in E(X, T)$  and  $\varphi \in C(X, \mathbb{R}) $, then
$$
h_{\mu}(T)+\int \varphi \mathrm{d} \mu=\lim _{\varepsilon \rightarrow 0} \liminf _{n \rightarrow \infty} \frac{1}{n} \log S_{\bar{d}_{n}}(\varphi, \mu, \varepsilon)=\lim _{\varepsilon \rightarrow 0} \limsup _{n \rightarrow \infty} \frac{1}{n} \log S_{\bar{d}_{n}}(\varphi, \mu, \varepsilon) .
$$
\end{theorem}

\begin{remark}
Notice that Theorem \ref{2.4} is different with \cite[Theorem 1.1]{huang2019measure1}.
\end{remark}

\section{Measure-theoretic pressure associated with modified Feldman-katok metric $d_{F K_{n}^{q}}$}

In this section, we will study measure-theoretic pressure associated with modified Feldman-katok metric $d_{F K_{n}^{q}}$.  Now we recall that the edit distance  $\bar{f}_{n}$  is defined by:
$$
\bar{f}_{n}\left(x_{0} x_{1} \ldots x_{n-1}, y_{0} y_{1} \ldots y_{n-1}\right)=1-\frac{k}{n}
$$
where  $k $ is the largest integer such that there exist
$
0 \leq i_{1}<\ldots<i_{k} \leq n-1,0 \leq j_{1}<\ldots<j_{k} \leq n-1
$
and  $x_{i_{s}}=y_{j_{s}} $ for $ s=1, \ldots, k .$

 Let  $\mu \in M(X, T) $ and  $\xi=\left\{A_{1}, \ldots, A_{|\xi|}\right\}$  be a finite partition of $ X $. We can identify the elements in  $\bigvee_{i=0}^{n-1} T^{-i} \xi$  and  $\{1, \ldots,|\xi|\}^{n}$  by
$
\bigcap_{i=0}^{n-1} T^{-i} A_{t_{i}}=\left(t_{0}, \ldots, t_{n-1}\right) 
$. Thus, for  $w \in\{1, \ldots,|\xi|\}^{n}$,  we can talk about  $\mu(w)$  and for  $A, B \in \bigvee_{i=0}^{n-1} T^{-i} \xi$  we can talk about  $\bar{f}_{n}(A, B) $. For $ x \in X $, the element of  $\xi$  containing  $x$  is denoted by  $\xi(x) $. We can simply replace $\bigvee_{i=0}^{n-1} T^{-i }\xi$ with  $\xi^{n}$. 

In 2023, Cai et al. \cite{cai2023feldman1} studied measure-theoretic entropy for ergodic measures in Feldman-Katok  metric, and they got a  measure-theoretic entropy formula as follows.
\begin{theorem}\cite{cai2023feldman1} \label{cai2023}
	Let  $(X, T)$  be a TDS. For any  $\omega		 \in {E}(X, T)$, then
	$$
	h_{\omega		}(T)=\lim _{\epsilon \rightarrow 0} \liminf _{n \rightarrow \infty} \frac{1}{n} \log {S}_{d_{F K_{n}}}(0, \omega		, \epsilon)=\lim _{\epsilon \rightarrow 0} \limsup _{n \rightarrow \infty} \frac{1}{n} \log {S}_{d_{F K_{n}}}(0, \omega		, \epsilon).
	$$
\end{theorem}

Next, we shall explore measure-theoretic pressure for ergodic measures in Feldman-Katok  metric, and we obtain a  measure-theoretic pressure formula as follows.
\begin{theorem} \label{3.2}
 Let  $(X, T)$  be a TDS. For any  $\mu \in E(X, T)$  and  $\varphi \in C(X, \mathbb{R}) $, then
$$
h_{\mu}(T)+\int \varphi \mathrm{d} \mu=\lim _{\varepsilon \rightarrow 0} \liminf _{n \rightarrow \infty} \frac{1}{n} \log S_{d_{FK_{n}}}(\varphi, \mu, \varepsilon)=\lim _{\varepsilon \rightarrow 0} \limsup _{n \rightarrow \infty} \frac{1}{n} \log S_{d_{FK_{n}}}(\varphi, \mu, \varepsilon) .
$$
\end{theorem}

\begin{proof}
 Step 1: First we show that for any  $\mu \in E(X, T)$,
$$
\lim _{\varepsilon \rightarrow 0} \limsup _{n \rightarrow \infty} \frac{1}{n} \log S_{d_{FK_{n}}}(\varphi, \mu, \varepsilon) \leq h_{\mu}(T)+\int \varphi \mathrm{d} \mu .
$$

Given  $\varepsilon >0$ and $\mu \in E(X, T)$.  It is obvious that  $d_{F K_{n}}(x, y) \leq {d}_{n}(x, y)$  and then  $B_{d_{n}}(x, \varepsilon) \subset B_{d_{F K_{n}}}(x, \varepsilon) $.   For any  $\tau>0$  and  sufficiently large $n$, according to Theorem 2.3, there exists a finite subset  $F_{n}$  of  $X$  such that
$$
\sum_{x_{i} \in F_{n}} e^{S_{n} \varphi\left(x_{i}\right)} \leq e^{n\left(h_{\mu}(T)+\int \varphi \mathrm{d} \mu+\tau\right)},
$$
and there is an open subset $ K_{n}$  of $ X $ with  $\mu\left(K_{n}\right)>1-\varepsilon$  such that for any  $x \in K_{n} $ there is  $y \in F_{n} $ with  $d_{n}(x, y)<\varepsilon $, which  yields that  $d_{F K_{n}}(x, y)\leq d_{n}(x, y)<\varepsilon $. It  follows that
$$
S_{d_{F K_{n}}}(\varphi, \mu, \varepsilon) \leq  \sum_{x_{i} \in F_{n}} e^{S_{n} \varphi\left(x_{i}\right)} \leq e^{n\left(h_{\mu}(T)+\int \varphi \mathrm{d} \mu+\tau\right)},
$$
and hence we obtain
$$
\begin{aligned}
	\lim _{\varepsilon \rightarrow 0} \limsup _{n \rightarrow \infty} \frac{1}{n} \log S_{d_{FK_{n}}}(\varphi, \mu, \varepsilon) & \leq \lim _{\varepsilon \rightarrow 0} \limsup _{n \rightarrow \infty} \frac{1}{n} \log S_{d_{n}}(\varphi, \mu, \varepsilon) \\
	& \leq h_{\mu}(T)+\int \varphi \mathrm{d} \mu+\tau.
\end{aligned}
$$

Let  $\tau \rightarrow 0 $, one has $$
\lim _{\varepsilon \rightarrow 0} \limsup _{n \rightarrow \infty} \frac{1}{n} \log S_{d_{FK_{n}}}(\varphi, \mu, \varepsilon) \leq h_{\mu}(T)+\int \varphi \mathrm{d} \mu .
$$
Step 2: Now we shall state that for any  $\mu \in E(X, T)$, then 
$$
h_{\mu}(T)+\int \varphi \mathrm{d} \mu \leq\lim _{\varepsilon \rightarrow 0} \liminf _{n \rightarrow \infty} \frac{1}{n} \log S_{d_{FK_{n}}}(\varphi, \mu, \varepsilon).
$$

Given a Borel partition  $\eta=\left\{A_{1}, \ldots, A_{k}\right\} $ of  $X$  and  $0<\delta<1 $, it is sufficient to show
$$
h_{\mu}(T, \eta)+\int \varphi \mathrm{d} \mu \leq \lim _{\varepsilon \rightarrow 0} \liminf _{n \rightarrow \infty} \frac{1}{n} \log S_{d_{FK_{n}}}(\varphi, \mu, \varepsilon)+2 \delta.
$$
Take $ 0<r<\frac{1}{2}$  with
$$
-4 r \log 2 r-2(1-2 r) \log (1-2 r)+4 r \log (k+1)+2 r\|\varphi\|<\delta .
$$

According to  \cite{walters1982introduction}  Lemma 4.15 and Corollary 4.12.1,  there exists a Borel partition  $\xi=\left\{B_{1}, \ldots, B_{k}, B_{k+1}\right\}$  of  $X$  associated with  $\eta$  such that  $B_{i} \subset A_{i}$  is closed for  $1 \leq i \leq k$  and  $B_{k+1}=X \backslash K $ with  $K=\bigcup_{i=1}^{k} B_{i}, \mu\left(B_{k+1}\right)<r^{2} $ and
$$
h_{\mu}(T, \eta)+\int \varphi \mathrm{d} \mu \leq h_{\mu}(T, \xi)+\delta+\int \varphi \mathrm{d} \mu.
$$

It is certain that  $B_{i} \cap B_{j}=\emptyset, 1 \leq i<j \leq k$  and  $\mu(K)>1-r^{2} $. Let  $b= \min _{1 \leq i<j \leq k} d\left(B_{i}, B_{j}\right) $. Then  $b>0 $. Given  $\gamma>0,0<\varepsilon<\min \left\{1-2 r, \frac{1}{2} b(1-2 r)\right\} $. For  $ n \in \mathbb{N}$  and  $\gamma \geq 0 $, by the definition of  $S_{d_{FK_{n}}}(\varphi, \mu, \varepsilon)$,  there are  $x_{1}, \ldots, x_{m(n)} \in X$  such that  $\mu\left(\bigcup_{i=1}^{m(n)} B_{d_{FK_{n}}}\left(x_{i}, \varepsilon\right)\right)>1-\varepsilon $, where
$$
\sum_{i=1}^{m(n)} e^{S_{n} \varphi\left(x_{i}\right)} \leq S_{d_{FK_{n}}}(\varphi, \mu, \varepsilon)+\gamma .
$$

Let  $F_{n}=\bigcup_{i=1}^{m(n)} B_{d_{FK_{n}}}(x_{i}, \varepsilon) $. Then  $\mu(F_{n})>1-\varepsilon $.  Take $	E_{n}=\left\{x \in X: \frac{\left|i\in [n]:T^ix\in K\right|}{n} \leq 1-r\right\} . $
Notice that 
$$\int_{X} \frac{\left|i\in [n]:T^ix\in K\right|}{n} \mathrm{~d} \mu(x)=\int_{X} \frac{1}{n} \sum_{i=0}^{n-1} 1_{K}\left(T^{i} x\right) \mathrm{d} \mu(x)=\mu(K)>1-r^{2} .$$
It follows that
$$
(1-r) \mu\left(E_{n}\right)+1-\mu\left(E_{n}\right) \geq \int_{X} \frac{\left|i\in [n]:T^ix\in K\right|}{n} \mathrm{~d} \mu(x)>1-r^{2},
$$
which indicates that  $\mu\left(E_{n}\right)<r $. Put  $W_{n}=\left(K \cap F_{n}\right) \backslash E_{n} $. Then
$$
\mu\left(W_{n}\right)>1-r^{2}-r-\varepsilon>1-2 r-\varepsilon
$$
and for  $z \in W_{n} $,
$$
\frac{\left|i\in [n]:T^iz\in K\right|}{n}>1-r.
$$

{\bf Claim 3.2.1} Let $\xi_{n,i}=\{A \in \xi^{n}: A \cap B_{d_{F K_{n}}}(x_{i}, \varepsilon) \cap W_{n} \neq \emptyset\}$, then
$$
|\xi_{n,i}| \leq(C_{n}^{[(1-2r-2 \varepsilon) n]})^{2} \cdot(k+1)^{[(2 r+2 \varepsilon) n]+1} .$$

{\bf Proof of claim 3.2.1}  Let
$
x \in A_{1} \cap B_{d_{F K_{n}}}\left(x_{i}, \varepsilon\right) \cap W_{n}$,  $y \in A_{2} \cap B_{d_{F K_{n}}}\left(x_{i}, \varepsilon\right) \cap W_{n},
$
where  $A_{1}, A_{2} \in \xi^{n} $. Then  $d_{F K_{n}}(x, y)<2 \varepsilon<b $. Denote
$$
D_{x}=\left\{0 \leq j \leq n-1: T^{j} x \in K\right\}, D_{y}=\left\{0 \leq j \leq n-1: T^{j} y \in K\right\} .
$$
Then
$
\left|D_{x}\right|>(1-r) n,\left|D_{y}\right|>(1-r) n .
$ Since  $d_{F K_{n}}(x, y)<2 \varepsilon $, let  $\pi$  be an  $(n, 2 \varepsilon)$-match of  $x$  and  $y$  with  $|\pi|>(1-2 \varepsilon) n $. Note that
$$
\left|\pi^{-1}\left(\pi\left(D(\pi) \cap D_{x}\right) \cap D_{y}\right)\right|>(1-2r-2 \varepsilon) n .
$$

Let  $\tilde{D}=\pi^{-1}\left(\pi\left(D(\pi) \cap D_{x}\right) \cap D_{y}\right) $, and  then  $\left|\tilde{D}\right|>(1-2 r-2 \varepsilon) n $. For every  $j \in \tilde{D}$, we have
$$
d\left(T^{j} x, T^{\pi(j)} y\right)<2 \varepsilon<b,
$$
and  then $T^{j} x, T^{\pi(j)} y \in K $. This implies that  $T^{j} x, T^{\pi(j)} y$  must be in one of the same element of $ \left\{B_{1}, \ldots, B_{k}\right\} $. It follows that
$
\bar{f}_{n}\left(A_{1}, A_{2}\right) \leq 2 r+2 \varepsilon.
$
Notice that the number of  $A $ satisfying
$
\bar{f}_{n}\left(A_{1}, A\right) \leq 2 r+2 \varepsilon
$
is not more than
$$
\left(C_{n}^{[(1-2 r-2 \varepsilon) n]}\right)^{2} \cdot(k+1)^{[(2 r+2 \varepsilon) n]+1} .
$$
We finish the claim 3.2.1.

Step 3: Lastly, we shall deal with the remaining part,  $$h_{\mu}(T, \eta)+\int \varphi \mathrm{d} \mu \leq \lim _{\varepsilon \rightarrow 0} \liminf _{n \rightarrow \infty} \frac{1}{n} \log S_{d_{FK_{n}}}(\varphi, \mu, \varepsilon)+2 \delta.$$
Take
$
\lambda(G):=\sup \left\{\left(S_{n} \varphi\right)(x): x \in G\right\} .$
For each  $G \in \xi^n \cap W_{n} $, we can choose some  $x_{G} \in \overline{G} $ so that  $S_{n} \varphi\left(x_{G}\right)=\lambda(G) $. For  $ n \in \mathbb{N} $, one has
$$
\begin{aligned}
	H_{\mu}(\xi^n )+\int S_{n} \varphi \mathrm{d} \mu
	\leq H_{\mu}(\xi^n  \vee\{W_{n}, X \backslash W_{n}\})+\int_{X} S_{n} \varphi \mathrm{d} \mu \\
	\leq-\mu(W_{n}) \sum_{U \in \xi^n} \frac{\mu(U \cap W_{n})}{\mu(W_{n})} \log \frac{\mu(U \cap W_{n})}{\mu(W_{n})}
	-\mu(X \backslash W_{n}) \sum_{U \in \xi^n } \frac{\mu(U \cap(X \backslash W_{n}))}{\mu\left(X \backslash W_{n}\right)} \log \frac{\mu(U \cap(X \backslash W_{n}))}{\mu\left(X \backslash W_{n}\right)} \\
	+\int_{W_{n}} S_{n} \varphi(x) \mathrm{d} \mu+\int_{X \backslash W_{n}} S_{n} \varphi(x) \mathrm{d} \mu-\mu(W_{n}) \log \mu(W_{n})
	-(1-\mu(W_{n})) \log (1-\mu(W_{n})) \\
	\leq \mu(W_{n}) H_{\mu_{W_{n}}}(\xi^n  \cap W_{n})+\int_{W_{n}} S_{n} \varphi(x) \mathrm{d} \mu
	+\mu(X \backslash W_{n}) \log \left|\xi^n \right|+\mu(X \backslash W_{n}) n\|\varphi\|+\log 2\\
	\leq H_{\mu_{W_{n}}}(\xi^n  \cap W_{n})+\int S_{n} \varphi(x) d \mu_{W_{n}}
	+\mu(X \backslash W_{n}) \log |\xi^n |+\mu(X \backslash W_{n}) n\|\varphi\|+\log 2 \\
	\leq \sum_{G \in \xi^n  \cap W_{n}} \mu_{W_{n}}(G)[-\log (\mu_{W_{n}}(G)+\lambda(G)]
	+\mu(X \backslash W_{n})(\log (k+1)^{n}+n\|\varphi\|)+\log 2 \\
	\leq \log \sum_{G \in \xi^n  \cap W_{n}} e^{\lambda(G)}+n(2 r+\varepsilon)(\log (k+1)+\|\varphi\|)+\log 2 \\
	\leq \log \sum_{0 \leq i \leq m(n)} \sum_{G \in \xi_{n, i}} e^{\lambda(G)}+n(2 r+\varepsilon)(\log (k+1)+\|\varphi\|)+\log 2. \\
\end{aligned}
$$

For any  $G \in \xi^{n}  \cap W_{n} $, then  $G \subset \bigcup_{i=1}^{m(n)} B_{d_{FK_{n}}}\left(x_{i}, \varepsilon\right) $, there exists  $1 \leq i \leq m(n)$  such that  $d_{FK_{n}}\left(x_{G}, x_{i}\right) \leq \varepsilon <2\varepsilon $. Let  $\tilde{\pi}$  be an  $(n,  2\varepsilon)$-match of  $x_{i}$  and  $x_{G}$  with  $|\tilde{\pi}|>(1-2\varepsilon) n $ and   for every  $j \in D({\tilde{\pi}}) $, one has
$
d\left(T^{j} x_{i}, T^{\pi(j)} x_{G}\right)<2 \varepsilon.
$
It follows that
$$
\lambda(G)=S_{n} \varphi\left(x_{G}\right) \leq S_{n} \varphi\left(x_{i}\right)+n \operatorname{Var}(\varphi, 2\varepsilon)+2n \varepsilon\|\varphi\| .
$$
Then
$$
\begin{aligned}
	&H_{\mu}(  \left.\xi^{n} \right)+\int_{X} S_{n} \varphi \mathrm{d} \mu \\
	\leq & \log \sum_{0 \leq i \leq m(n)} \sum_{G \in \xi_{n,i}} e^{S_{n} \varphi\left(x_{i}\right)+n \operatorname{Var}(\varphi, 2\varepsilon)+2n \varepsilon\|\varphi\|} +n(2r+\varepsilon)(\log (k+1)+\|\varphi\|)+\log 2 .
\end{aligned}
$$
Thus, we have
$$
\begin{aligned}
	& H_{\mu}\left(\xi^n \right)+n \int \varphi \mathrm{d} \mu = H_{\mu}\left(\xi^n\right)+\int S_{n} \varphi \mathrm{d} \mu \\
	\leq & \log \left(\left(C_{n}^{[(1-2 r-2 \varepsilon) n]}\right)^{2} \cdot(k+1)^{[(2 r+2 \varepsilon) n]+1}  \cdot \sum_{0 \leq i \leq m(n)} e^{S_{n} \varphi\left(x_{i}\right)+n \operatorname{Var}(\varphi, 2\varepsilon)+2n \varepsilon\|\varphi\|}\right) \\
	& +n(2 r+\varepsilon)(\log (k+1)+\|\varphi\|)+\log 2 .
\end{aligned}
$$
Hence,
$$
\begin{aligned}
	& \frac{1}{n} H_{\mu}\left(\xi^n \right)+\int \varphi \mathrm{d} \mu \\
	\leq & \frac{1}{n} \log \left(\sum_{0 \leq i \leq m(n)} e^{S_{n} \varphi\left(x_{i}\right)}\right)+\frac{1}{n} \log \left(C_{n}^{[(1-2 r-2 \varepsilon) n]}\right)^{2}+\frac{1}{n} \log (k+1)^{[(2 r+2 \varepsilon) n]+1} \\
	& +\operatorname{Var}(\varphi, 2\varepsilon)+2\varepsilon\|\varphi\|+(2 r+\varepsilon)(\log (k+1)+\|\varphi\|)+\frac{\log 2}{n} \\
	\leq & \frac{1}{n} \log \left(S_{d_{FK_{n}}}(\varphi, \mu, \varepsilon)+\gamma\right)+\frac{1}{n} \log (\left(C_{n}^{[(1-2 r-2 \varepsilon) n]}\right)^{2}+\frac{1}{n} \log (k+1)^{[(2 r+2 \varepsilon) n]+1} \\
	& +\operatorname{Var}(\varphi, 2\varepsilon)+2\varepsilon\|\varphi\|+(2 r+\varepsilon)(\log (k+1)+\|\varphi\|)+\frac{\log 2}{n}.
\end{aligned}
$$

Using the Stirling's formula,
$$
\lim _{n \rightarrow+\infty} \frac{1}{n} \log C_{n}^{[n (1-2 r-2 \varepsilon) ]}=-(1-2 r-2 \varepsilon)  \log (1-2 r-2 \varepsilon) -(2 r+2 \varepsilon) \log (2 r+2 \varepsilon),
$$
Furthermore,  taking  $\varepsilon \rightarrow 0 $, we can obtain 
$$
\begin{aligned}
	&\quad h_{\mu}(T, \xi)+\int \varphi \mathrm{d} \mu \\
	&\leq \lim _{\varepsilon \rightarrow 0} \liminf _{n \rightarrow \infty} \frac{1}{n} \log S_{d_{FK_{n}}}(\varphi, \mu, \varepsilon)\\
	&\quad-4 r \log 2 r-2(1-2 r) \log (1-2 r)+4 r \log (k+1)+2 r\|\varphi\|\\
	&<\lim _{\varepsilon \rightarrow 0} \liminf _{n \rightarrow \infty} \frac{1}{n} \log S_{d_{FK_{n}}}(\varphi, \mu, \varepsilon)+\delta .
\end{aligned}
$$
Therefore, we get that $$h_{\mu}(T, \eta)+\int \varphi \mathrm{d} \mu \leq h_{\mu}(T, \xi)+\int \varphi \mathrm{d} \mu+\delta<\lim _{\varepsilon \rightarrow 0} \liminf _{n \rightarrow \infty} \frac{1}{n} \log S_{d_{FK_{n}}}(\varphi, \mu, \varepsilon)+2 \delta .$$  The proof is finished.
\end{proof}

\begin{lemma}\cite{zhang2023note} \label{3.3}
 Let  $(X, T)$  be a TDS. For any  $\mu \in E(X, T)$ and  $\varphi \in C(X, \mathbb{R}) $, then
$$
\int \varphi \mathrm{d} \mu+h_{\mu}(T)=\lim _{\varepsilon \rightarrow 0} \liminf _{n \rightarrow \infty} \frac{1}{n} \log \bar{S}_{n}(\mu, \varphi, \varepsilon)=\lim _{\varepsilon \rightarrow 0} \limsup _{n \rightarrow \infty} \frac{1}{n} \log \bar{S}_{n}(\mu, \varphi, \varepsilon)
$$
and
$$
\int \varphi \mathrm{d} \mu+h_{\mu}(T)=\lim _{\varepsilon \rightarrow 0} \liminf _{n \rightarrow \infty} \frac{1}{n} \log \tilde{S}_{n}(\mu, \varphi, \varepsilon)=\lim _{\varepsilon \rightarrow 0} \limsup _{n \rightarrow \infty} \frac{1}{n} \log \tilde{S}_{n}(\mu, \varphi, \varepsilon) .
$$
\end{lemma}

\begin{theorem} \label{3.4}
 Let  $(X, T)$  be a TDS. For any  $\mu \in E(X, T)$  and  $\varphi \in C(X, \mathbb{R}) $, then
$$
h_{\mu}(T)+\int \varphi \mathrm{d} \mu=\lim _{\varepsilon \rightarrow 0} \liminf _{n \rightarrow \infty} \frac{1}{n} \log S_{n}^{FK}(\varphi, \mu, \varepsilon)=\lim _{\varepsilon \rightarrow 0} \limsup _{n \rightarrow \infty} \frac{1}{n} \log S_{n}^{FK}(\varphi, \mu, \varepsilon) .
$$
\end{theorem}

\begin{proof}
 It follows from  Lemma \ref{3.3} and a similar approach in Theorem \ref{3.2}, which we leave  to the readers here.

\end{proof}

\section{Topolgical pressure  associated with modified Feldman-katok metric $d_{F K_{n}^{q}}$}

In this section, we  will investigate the topological pressure, now let us recall the notion of topological pressure. Let  $(X, T)$  be a TDS. Given $n \in \mathbb{N}$  and  $\varphi \in C(X, \mathbb{R}) $, we consider
$$
P_{n}(T, \varphi, \varepsilon)=\inf \left\{\sum_{x_{i} \in F} e^{S_{n} \varphi\left(x_{i}\right)} \mid F \text { is a }(n, \varepsilon)\text {-spanning set for } X\right\} \text {. }
$$

For  $\varepsilon>0$  and  $\varphi \in C(X, \mathbb{R}) $, put
$$
P(T, \varphi, \varepsilon)=\limsup _{n \rightarrow \infty} \frac{1}{n} \log P_{n}(T, \varphi, \varepsilon) .
$$
When  $\varepsilon \rightarrow 0 $, we can get
$$
P_{\text {top }}(T, \varphi)=\lim _{\varepsilon \rightarrow 0} P(T, \varphi, \varepsilon) \text {. }
$$
Particularly, when $\varphi=0$,  it is well known that $P_{\text {top }}(T, 0)$ is just the classical topological entropy $h(T)$ which is defined in \cite{walters1982introduction}.

It is worth noting that the relationship between topological pressure and measure-theoretic pressure is stated as a fact in \cite{walters1982introduction} which is Theorem \ref{4.2} below, it plays an important role in the proof of Theorem \ref{4.3}.

Let  $(X, T)$  be a TDS. For  $q, n \in \mathbb{N}$  and  $\varepsilon>0 $,
$$
S_{d_{F K_{n}^{q}}}(\varphi, \varepsilon)=\inf \left\{\sum_{i=1}^{k} e^{S_{n, q}\varphi\left(x_{i}\right)} \text { s.t. } \bigcup_{i=1}^{k} B_{d_{F K_{n}^{q}}}\left(x_{i}, \varepsilon\right)=X\right\} .
$$
We define
$$
S_{n}^{FK}(\varphi, \varepsilon)=\inf _{q \geq 1} S_{d_{F K_{n}^{q}}}(\varphi, \varepsilon) .
$$

Similarly, we can define  $S_{\bar{d}_{n}^{q}}(\varphi, \varepsilon)$, $\bar{S}_{n}(\varphi, \varepsilon)$  by changing $\bar{d}_{n}^{q}$  into  $d_{FK_{n}^{q}}$ in the definition of $S_{d_{F K_{n}^{q}}}(\varphi, \varepsilon),S_{n}^{FK}(\varphi, \varepsilon)$ respectively.

 Cai et al. \cite{cai2023feldman} presented the notion of  the topological pressure with respect to Feldman-Katok metric, and  showed the topological pressure with respect to Feldman-Katok metric is equal to the classical topological pressure.

\begin{theorem}\cite{cai2023feldman} \label{4.1}
Let $ (X, T)$  be a TDS and  $\varphi \in C(X, \mathbb{R}) $. Then
$$
P_{\text {top }}(T, \varphi)=\lim _{\varepsilon \rightarrow 0} \limsup _{n \rightarrow \infty} \frac{1}{n} \log S_{d_{FK_{n}}}(\varphi, \varepsilon).
$$
\end{theorem}

\begin{theorem}\cite{walters1982introduction} \label{4.2}
	Let $ (X, T)$  be a TDS and  $\varphi \in C(X, \mathbb{R}) $. 
	then
	$$
	P_{\text {top }}(T, \varphi)=\sup \left\{h_{\mu}(T)+\int \varphi \mathrm{d} \mu: \mu \in E(X, T)\right\} .
	$$
\end{theorem}

\begin{theorem} \label{4.3}
Let $ (X, T)$  be a TDS and  $\varphi \in C(X, \mathbb{R}) $. Then
$$
P_{\text {top }}(T, \varphi)=\lim _{\varepsilon \rightarrow 0} \liminf _{n \rightarrow \infty} \frac{1}{n} \log S_{n}^{FK}(\varphi, \varepsilon)=\lim _{\varepsilon \rightarrow 0} \limsup _{n \rightarrow \infty} \frac{1}{n} \log S_{n}^{FK}(\varphi, \varepsilon)
$$
\end{theorem}

\begin{proof}
 By Theorem \ref{4.1},  it is clear that $$\lim _{\varepsilon \rightarrow 0} \limsup _{n \rightarrow \infty} \frac{1}{n} \log  S_{n}^{FK}(\varphi, \varepsilon)\leq \lim _{\varepsilon \rightarrow 0} \limsup _{n \rightarrow \infty} \frac{1}{n} \log S_{d_{FK_{n}}}(\varphi, \varepsilon)=P_{\text {top }}(T, \varphi).$$

Next, we just suffer to show that 
$$
P_{\text {top }}(T, \varphi) \leq \lim _{\varepsilon \rightarrow 0} \liminf _{n \rightarrow \infty} \frac{1}{n} \log S_{n}^{FK}(\varphi, \varepsilon) .
$$
Using Theorem \ref{4.2}, then we need to  claim that for any  $\mu \in E(X, T)$, one has 
$$
h_{\mu}(T)+\int \varphi \mathrm{d} \mu \leq \lim _{\varepsilon \rightarrow 0} \liminf _{n \rightarrow \infty} \frac{1}{n} \log S_{n}^{FK}(\varphi, \varepsilon) .
$$

It is easy to check the fact that for  $\varepsilon>0$  and  $n \in \mathbb{N} $, $
S_{n}^{FK}(\varphi, \mu, \varepsilon) \leq S_{n}^{FK}(\varphi, \varepsilon) .
$
Afterwards, according to Theorem \ref{3.4}, then we obtain
$$
h_{\mu}(T)+\int \varphi \mathrm{d} \mu=\lim _{\varepsilon \rightarrow 0} \liminf _{n \rightarrow \infty} \frac{1}{n} \log S_{n}^{FK}(\varphi, \mu, \varepsilon) \leq \lim _{\varepsilon \rightarrow 0} \liminf _{n \rightarrow \infty} \frac{1}{n} \log S_{n}^{FK}(\varphi, \varepsilon) .
$$
Hence, we get that 
$$
P_{\text {top }}(T, \varphi)=\lim _{\varepsilon \rightarrow 0} \liminf _{n \rightarrow \infty} \frac{1}{n} \log S_{n}^{FK}(\varphi, \varepsilon)=\lim _{\varepsilon \rightarrow 0} \limsup _{n \rightarrow \infty} \frac{1}{n} \log S_{n}^{FK}(\varphi, \varepsilon)
$$
We finish the proof.
\end{proof}

\section{Measure-theoretic pressure and topolgical pressure associated with modified max-mean metric $\hat{d}_{n}^{q}$}

In this section, we will study measure-theoretic pressure and topological pressure associated with modified max-mean metric. Now we recall three metrics on $ X$. Given  $n \in \mathbb{N} $,  $x, y \in X $, let
$$
d_{n}(x, y)=\max \left\{d\left(T^{i} x, T^{i} y\right): 0 \leq i \leq n-1\right\}$$ 
$$\quad \bar{d}_{n}(x, y)=\frac{1}{n} \sum_{i=0}^{n-1} d\left(T^{i} x, T^{i} y\right)
$$
and
$$
\hat{d}_{n}(x, y)=\max \left\{\bar{d}_{k}(x, y): 1 \leq k \leq n\right\} .
$$

It is clear that, for all  $x, y \in X $,
$$
\bar{d}_{n}(x, y)\leq \hat{d}_{n}(x, y) \leq d_{n}(x, y).
$$

Next, we shall introduce the concept of  modified  max-mean metric. Let  $(X, T) $ be a TDS. Given  $\varphi \in C(X, \mathbb{R}), \mu \in M(X, T) $, $\varepsilon>0 $, $x, y \in X$  and  $n,q \in \mathbb{N}$, let $$
\hat{d}_{n}^{q}(x, y)=\max \left\{\bar{d}_{k}^{q}(x, y): 1 \leq k \leq n\right\}
$$
$$
S_{\hat{d}_{n}^{q}}(\varphi, \mu, \varepsilon)=\inf \left\{\sum_{i=1}^{k} e^{S_{n, q} \varphi\left(x_{i}\right)}, \text { s.t. } \mu\left(\bigcup_{i=1}^{k} B_{\hat{d}_{n}^{q}}\left(x_{i}, \varepsilon\right)\right)>1-\varepsilon\right\}
$$
where
$$
B_{\hat{d}_{n}^{q}}(x, \varepsilon)=\left\{y \in X: \hat{d}_{n}^{q}(x, y)<\varepsilon\right\}
$$
for any  $x \in X $. We define
$$
\hat{S}_{n}(\varphi, \mu, \varepsilon)=\inf _{q \geq 1} S_{\hat{d}_{n}^{q}}(\varphi, \mu, \varepsilon) .
$$
Similarly, we can define  $S_{\hat{d}_{n}^{q}}(\varphi, \varepsilon)$, $\hat{S}_{n}(\varphi, \varepsilon)$  by changing   $\hat{d}_{n}^{q}$  into  $d_{FK_{n}^{q}}$ in the definition of $S_{d_{F K_{n}^{q}}}(\varphi, \varepsilon),S_{n}^{FK}(\varphi, \varepsilon)$ respectively.

\begin{lemma} \label{5.1}
\cite{zhang2023note} Let $ (X, T)$  be a TDS and  $\varphi \in C(X, \mathbb{R}) $. Then
$$
P_{\text {top }}(T, \varphi)=\lim _{\varepsilon \rightarrow 0} \liminf _{n \rightarrow \infty} \frac{1}{n} \log \bar{S}_{n}(\varphi, \varepsilon)=\lim _{\varepsilon \rightarrow 0} \limsup _{n \rightarrow \infty} \frac{1}{n} \log \bar{S}_{n}(\varphi, \varepsilon)
$$
and
$$
P_{\text {top }}(T, \varphi)=\lim _{\varepsilon \rightarrow 0} \liminf _{n \rightarrow \infty} \frac{1}{n} \log \tilde{S}_{n}(\varphi, \varepsilon)=\lim _{\varepsilon \rightarrow 0} \limsup _{n \rightarrow \infty} \frac{1}{n} \log \tilde{S}_{n}(\varphi, \varepsilon) .
$$
\end{lemma}

\begin{lemma}\label{5.2}
\cite{zhang2023note}  Let  $(X, T)$  be a TDS. Then for any  $\mu \in E(X, T)$  and  $\varphi \in C(X, \mathbb{R}) $, we have
$$
\int \varphi \mathrm{d} \mu+h_{\mu}(T)=\lim _{\varepsilon \rightarrow 0} \liminf _{n \rightarrow \infty} \frac{1}{n} \log \bar{S}_{n}(\mu, \varphi, \varepsilon)=\lim _{\varepsilon \rightarrow 0} \limsup _{n \rightarrow \infty} \frac{1}{n} \log \bar{S}_{n}(\mu, \varphi, \varepsilon)
$$
and
$$
\int \varphi \mathrm{d} \mu+h_{\mu}(T)=\lim _{\varepsilon \rightarrow 0} \liminf _{n \rightarrow \infty} \frac{1}{n} \log \tilde{S}_{n}(\mu, \varphi, \varepsilon)=\lim _{\varepsilon \rightarrow 0} \limsup _{n \rightarrow \infty} \frac{1}{n} \log \tilde{S}_{n}(\mu, \varphi, \varepsilon) .
$$
\end{lemma}

Hence, according to Lemma \ref{5.1}, Lemma \ref{5.2} and  the fact $\bar{d}_{n}^{q}(x, y)\leq \hat{d}_{n}^{q}(x, y) \leq d_{n}^{q}(x, y)$,   we can derive the following conclusions, here we leave them to readers.

\begin{theorem}
  Let  $(X, T)$  be a TDS. For any  $\mu \in E(X, T)$  and  $\varphi \in C(X, \mathbb{R}) $, then
$$
\int \varphi \mathrm{d} \mu+h_{\mu}(T)=\lim _{\varepsilon \rightarrow 0} \liminf _{n \rightarrow \infty} \frac{1}{n} \log \hat{S}_{n}(\mu, \varphi, \varepsilon)=\lim _{\varepsilon \rightarrow 0} \limsup _{n \rightarrow \infty} \frac{1}{n} \log \hat{S}_{n}(\mu, \varphi, \varepsilon).
$$
\end{theorem}

\begin{theorem}
	 Let  $(X, T)$  be a TDS. For any  $\mu \in E(X, T)$  and  $\varphi \in C(X, \mathbb{R}) $, then
	$$
	\int \varphi \mathrm{d} \mu+h_{\mu}(T)=\lim _{\varepsilon \rightarrow 0} \liminf _{n \rightarrow \infty} \frac{1}{n} \log {S}_{\hat{d}_{n}}(\mu, \varphi, \varepsilon)=\lim _{\varepsilon \rightarrow 0} \limsup _{n \rightarrow \infty} \frac{1}{n} \log {S}_{\hat{d}_{n}}(\mu, \varphi, \varepsilon).
	$$
\end{theorem}

\begin{theorem}
 Let  $(X, T)$  be a TDS. For any  $\mu \in E(X, T)$  and  $\varphi \in C(X, \mathbb{R}) $, then
	$$
	P_{\text {top }}(T, \varphi)=\lim _{\varepsilon \rightarrow 0} \limsup _{n \rightarrow \infty} \frac{1}{n} \log {S}_{\hat{d}_{n}}(\varphi, \varepsilon)=\lim _{\varepsilon \rightarrow 0} \liminf _{n \rightarrow \infty} \frac{1}{n} \log {S}_{\hat{d}_{n}}(\varphi, \varepsilon).
	$$
\end{theorem}

\begin{theorem}
	 Let  $(X, T)$  be a TDS. For any  $\mu \in E(X, T)$  and  $\varphi \in C(X, \mathbb{R}) $, then
	$$
	P_{\text {top }}(T, \varphi)=\lim _{\varepsilon \rightarrow 0} \liminf _{n \rightarrow \infty} \frac{1}{n} \log \hat{S}_{n}(\varphi, \varepsilon)=\lim _{\varepsilon \rightarrow 0} \limsup _{n \rightarrow \infty} \frac{1}{n} \log \hat{S}_{n}(\varphi, \varepsilon).
	$$
\end{theorem}

\section{Appendix}
We will give a results about Brin-Katok formula associated with max-mean metric. It is well known  that  Brin and Katok \cite{brin2006local} introduced the Brin-Katok formula: that is, let  $(X, T)$  be a TDS and  $\mu \in M(X, T) $, then for  $\mu$-a.e.  $x \in X$,  one has
$$
h_{\mu}(T, x):=\lim _{\delta \rightarrow 0} \liminf _{n \rightarrow \infty}-\frac{\log \mu\left(B_{d_{n}}(x, \delta)\right)}{n}=\lim _{\delta \rightarrow 0} \limsup _{n \rightarrow \infty}-\frac{\log \mu\left(B_{d_{n}}(x, \delta)\right)}{n},
$$
where  $h_{\mu}(T, x)$  is  $T $-invariant and  $\int h_{\mu}(T, x) d x=h_{\mu}(T) $. 

When  Bowen metric is replaced  by Feldman-Katok metric, one has

{\bf Theorem A.1} \cite{cai2023feldman1} Let  $(X, T)$  be a TDS and  $\mu \in M(X, T) $. Then for  $\mu$-a.e. $x \in X $,
$$
h_{\mu}(T, x)=\lim _{\delta \rightarrow 0} \liminf _{n \rightarrow \infty}-\frac{\log \mu\left(B_{d_{F K_{n}}}(x, \delta)\right)}{n}=\lim _{\delta \rightarrow 0} \limsup _{n \rightarrow \infty}-\frac{\log \mu\left(B_{d_{F K_{n}}}(x, \delta)\right)}{n} .
$$

When Bowen metric is replaced  by mean metric, one has

{\bf Theorem A.2} \cite{huang2018entropy} Let  $(X, T)$  be a TDS and  $\mu \in M(X, T) $. Then for  $\mu$-a.e. $x \in X $,
$$
h_{\mu}(T, x)=\lim _{\delta \rightarrow 0} \liminf _{n \rightarrow \infty}-\frac{\log \mu\left(B_{\bar{d}_{n}}(x, \delta)\right)}{n}=\lim _{\delta \rightarrow 0} \limsup _{n \rightarrow \infty}-\frac{\log \mu\left(B_{\bar{d}_{n}}(x, \delta)\right)}{n} .
$$

Hence, with the fact $
\bar{d}_{n}(x, y)\leq \hat{d}_{n}(x, y) \leq d_{n}(x, y),
$ this yields that $B_{d_{n}}(x, \delta)\subset B_{\hat{d}_{n}}(x, \delta)\subset B_{\bar{d}_{n}}(x, \delta)$ for any $\delta>0$, hence,  we can get the  following conclusion.

{\bf Theorem A.3} Let  $(X, T)$  be a TDS and  $\mu \in M(X, T) $. Then for  $\mu$-a.e. $x \in X $, one has
$$
h_{\mu}(T, x)=\lim _{\delta \rightarrow 0} \liminf _{n \rightarrow \infty}-\frac{\log \mu\left(B_{\hat{d}_{n}}(x, \delta)\right)}{n}=\lim _{\delta \rightarrow 0} \limsup _{n \rightarrow \infty}-\frac{\log \mu\left(B_{\hat{d}_{n}}(x, \delta)\right)}{n} .
$$


\bigskip \noindent{\bf Acknowledgement} The authors are grateful to the referee for a careful reading of the paper and a multitude of corrections and helpful suggestions.

\bigskip \noindent{\bf Data availability}
Data sharing not applicable to this article as no datasets were generated or
analysed during the current study.

\bigskip \noindent{\bf Funding} No funding.

\bigskip \noindent{\bf Declarations}

\bigskip \noindent{\bf Conflict of interest}
The authors have no financial or proprietary interests in any material discussed
in this article.

\bigskip \noindent{\bf Ethical Approval}
Not applicable.

\end{document}